\documentclass[10pt]{amsart}
\def\cal#1{\mathcal{#1}}

          \usepackage{amssymb}
          \usepackage{pgf} 
          \usepackage{amsmath}
          \usepackage{amsthm}
          \usepackage{amsfonts}
          \usepackage{enumerate}
          \usepackage{url}
          \usepackage{color}

\usepackage[numbers,sort&compress]{natbib}
\def\cal{\mathcal}
\newcommand{\comment}[1]{}
\newcommand{\ind}{{\bf 1}}

\newcommand{\proba}{\mathbb P}
\newcommand{\esp}{{\mathbb E}}

\newcommand{\cov}{{\rm{Cov}}}
\newcommand{\var}{{\rm{Var}}}

\newcommand{\calJ}{{\cal J}}

\newcommand{\calX}{{\cal X}}

\newcommand{\eqnh}{\begin{eqnarray*}}
\newcommand{\eqne}{\end{eqnarray*}}
\newcommand{\eqnhn}{\begin{eqnarray}}
\newcommand{\eqnen}{\end{eqnarray}}
\newcommand{\equh}{\begin{equation}}
\newcommand{\eque}{\end{equation}}

\def\summ#1#2#3{\sum_{#1 = #2}^{#3}}
\def\prodd#1#2#3{\prod_{#1 = #2}^{#3}}

\def\topp#1{^{(#1)}}

\def\abs#1{\left|#1\right|}

\def\ccbb#1{\left\{#1\right\}}

\def\pp#1{\left(#1\right)} 
 
\def\bb#1{\left[#1\right]}

\def\mmid{\;\middle\vert\;}

\def\ceil#1{\left\lceil #1 \right\rceil}

\def\qmand{\quad\mbox{ and }\quad}

\def\qmwith{\quad\mbox{ with }\quad}
\def\mfa{\mbox{ for all }}

\def\wt#1{\widetilde{#1}}

\def\limn{\lim_{n\to\infty}}
\def\limm{\lim_{m\to\infty}}

\def\limsupn{\limsup_{n\to\infty}}

\def\R{{\mathbb R}}

\def\N{{\mathbb N}}

\usepackage{epsfig}
\usepackage{ifthen}

\newtheorem{Thm}{Theorem}[section]
\newtheorem{Lem}[Thm]{Lemma}
\newtheorem{Prop}[Thm]{Proposition}

\theoremstyle{definition}

\numberwithin{equation}{section}

\usepackage{color,soul}
\renewcommand{\comment}[1]{{\color{blue}\fbox{#1}}}
\def\commentYZ#1{}

\newcommand{\longcommenthide}[1]{}

\title{Large jumps of $q$-Ornstein--Uhlenbeck processes}
\longcommenthide{\author{W\l odzimierz Bryc}
\address
{
W\l odzimierz Bryc\\
Department of Mathematical Sciences\\
University of Cincinnati\\
2815 Commons Way\\
Cincinnati, OH, 45221-0025, USA.
}
\email{wlodzimierz.bryc@uc.edu}
}
\author{Yizao Wang}
\address
{
Yizao Wang\\
Department of Mathematical Sciences\\
University of Cincinnati\\
2815 Commons Way\\
Cincinnati, OH, 45221-0025, USA.
}
\email{yizao.wang@uc.edu}
\begin{document}\sloppy

\date{\today. File: \jobname.tex}


\begin{abstract}We continue the investigation of sample paths of $q$-Ornstein--Uhlenbeck processes. We show that for each $q\in(-1,1)$, the process has big jumps crossing from near one end point of the domain to the other with positive probability. Moreover, the number of such jumps in an appropriately enlarged window converges weakly to a Poisson random variable. 
\end{abstract}
\maketitle

\section{Introduction}
In this paper, we continue the investigation of path properties of $q$-Gaussian processes started in \citep{bryc16local}.
We focus on the so-called $q$-Ornstein--Uhlenbeck process for $q\in(-1,1)$. These are stationary classical Markov processes corresponding to the processes with the same name  arising first from non-commutative probability \citep{biane98processes,bozejko97qGaussian}. 
In classical probability theory, they turned out to be closely related to the so-called quadratic harnesses, which are continuous-time stochastic processes characterized by expressions of conditional mean and variance given the past and future \citep{bryc05conditional,bryc07quadratic}.

The present paper takes a purely probabilistic point of view of $q$-Ornstein--Uhlenbeck processes, and in particular studies their path properties.
There exists already a vast literature on path properties of stochastic processes on  various aspects, including regularity of sample paths, fractal properties, extremes, excursions and overshoots. The extensively investigated families of processes include notably Gaussian processes, L\'evy processes,  Markov processes and stable processes, among others. See for example \citep{adler90introduction,xiao04random,morters10brownian,piterbarg96asymptotic,samorodnitsky94stable,bertoin96levy,xiao09sample} and references therein.
The $q$-Ornstein--Uhlenbeck processes, however, present some intriguing features that distinguish them from most of the well studied processes so far, as we shall see below. 

For each $q\in(-1,1)$, the  $q$-Ornstein--Uhlenbeck process is a stationary Markov process with domain $[-2/\sqrt{1-q},2/\sqrt{1-q}]$, and it has explicit probability density function and transition probability density function (see~\eqref{eq:p} and~\eqref{eq:pst} below). 
It is known that as $q\uparrow 1$, the process converges weakly to the standard Ornstein--Uhlenbeck  process, a stationary Gaussian Markov process with continuous sample paths. However, for each $q\in(-1,1)$, the sample paths are known to be discontinuous. \citet{szablowski12qwiener} first asked a series of questions on path properties of $q$-Ornstein--Uhlenbeck processes. In the previous paper \citep{bryc16local}, we showed that for each fixed $q\in(-1,1)$, locally the process has small jumps and behaves as a Cauchy process, in the framework of tangent processes \citep{falconer03local}. 
This paper continues and complements the investigation by looking at large jumps. We shall prove that with strictly positive probability, there are large jumps crossing from one side of the bounded support to the other side. We describe the asymptotic law of the number of such jumps in the form of a Poisson limit theorem. 
All our analysis are based on the explicit formula of transition probability density function, although the distribution of jumps cannot be derived from it directly.

We first review the local structure revealed in \citep{bryc16local}.
It is shown in \citep[Section 4]{szablowski12qwiener} that the $q$-Ornstein--Uhlenbeck process has a version in $D([0,\infty))$, the space of c\`adl\`ag functions on $[0,\infty)$.  
Let $X\topp q= \{X\topp q_t\}_{t\ge0}$ denote a $q$-Ornstein--Uhlenbeck process, and we assume $X\topp q\in D([0,\infty))$ with probability one. 
In \citep{bryc16local}, we showed that the local structure of $X\topp q$ can be characterized via the notion of tangent processes \citep{falconer03local}. Namely,  for all $x\in(-2/\sqrt{1-q},2/\sqrt{1-q})$, under the law $\proba(\cdot\mid X_0 = x)$, as $\epsilon\downarrow0$,
\[
\ccbb{\frac{X_{\epsilon t}-X_0}{\epsilon}}_{t\ge 0}
\]
converges weakly in $D([0,\infty))$ to a Cauchy process, with appropriate scaling depending on $x$ and $q$. For $x = \pm 2/\sqrt{1-q}$, under a different scaling $\epsilon^2$ instead of $\epsilon$, the limiting tangent process is a different self-similar process, closely related to so-called $1/2$-stable Biane process. 
In words, the tangent processes characterize the local small jumps of the original $q$-Ornstein--Uhlenbeck process. It is remarkable that although as $q\uparrow 1$ the $q$-Ornstein--Uhlenbeck process converges weakly to the standard Ornstein--Uhlenbeck process, which has continuous sample paths with probability one, for all $q\in(-1,1)$ the $q$-Ornstein--Uhlenbeck process has small local jumps.  

In this paper, we study big jumps of the same processes.  Namely, with $q$ fixed, we consider the number of large jumps from
within  
$\epsilon$-margin of the lower boundary to within 
$\epsilon$-margin of the upper boundary of the domain during the period $(a,b]$:
\[
N\topp q((a,b],\epsilon) := \abs{\ccbb{t\in(a,b]: X_{t-}\topp q<-\frac2{\sqrt{1-q}}+\epsilon, X_t\topp q>\frac2{\sqrt{1-q}}-\epsilon}}.
\]
Our main result is the following.
\begin{Thm}\label{thm:1}
For $q\in(-1,1)$, as $\epsilon\downarrow 0$, $N\topp q((0,\epsilon^{-3}],\epsilon)$ converges weakly to a Poisson random variable with parameter 
\[
\alpha_q:
= \frac1{18\pi^2}(1-q)^{3/2}\prodd k1\infty\frac{(1-q^k)^7}{(1+q^k)^4}.
\]

\end{Thm}
This result reveals another qualitative behavior of the $q$--Ornstein--Uhlenbeck process for all $q<1$: the process may have big jumps crossing almost the entire domain, although the frequency of such jumps decreases to zero as $q\uparrow 1$. At the same time, in contrast to the tangent process that characterizes local jumps, the theorem above characterizes big jumps by looking at the process globally: such big jumps can only be observed with non-negligible probability over an increasing time window, as over any fixed window the probability of having at least one such jump vanishes as $\epsilon\downarrow 0$: we will  first show that the order of this probability is $\epsilon^3$.

We prove the main result in two steps. In Section~\ref{sec:bigjumps} we compute the asymptotic probability of having at least one large jump during the interval $(0,1]$, by the double-sum method. In Section~\ref{sec:Poisson} we apply a Poisson limit theorem established by \citet{chen75poisson} for $\phi$-mixing sequence of random variables. 
These tools are nowadays standard techniques in extreme value theory for stationary processes, see for example \citep{aldous89probability,leadbetter83extremes}. However, applications to asymptotic probability of large jumps of stochastic processes, as the problem investigated in this paper, are rarely seen in the literature. We also point out that all the descriptions of jumps obtained so far, local and global, are asymptotic:  we do not know any precise descriptions yet as the distribution of jumps of L\'evy processes, or the distribution of passage across a given level of subordinators \citep{bertoin96levy}. Most questions raised originally in \citep{szablowski12qwiener}  remain still open.

We conclude the introduction with a few formulas on the $q$-Ornstein--Uhlenbeck process that will be used in the paper. 
For each $q\in(-1,1)$,  $X\topp q$ is a stationary Markov process with domain $[-2/\sqrt{1-q},2/\sqrt{1-q}]$. Its marginal  probability density function $p\topp q(x)$ and the  transition probability density function $p\topp q_{s,t}(x,y)$ are given by, for $x,y\in[-2/\sqrt{1-q},2/\sqrt{1-q}]$,
\equh\label{eq:p}
p\topp q(x) = \frac{\sqrt{1-q}\cdot(q)_\infty}{2\pi}\sqrt{4-(1-q)x^2}\prodd k1\infty\bb{(1+q^k)^2-(1-q)x^2q^k},
\eque
\equh\label{eq:pst}
p\topp q_{s,t}(x,y) = (e^{-2(t-s)};q)_\infty \prodd k0\infty\frac1{\varphi_{q,k}(t-s,x,y)} \cdot p(y),
\eque
with
\[
\varphi_{q,k}(\delta,x,y) = (1-e^{-2\delta}q^{2k})^2 - (1-q)e^{-\delta}q^k(1+e^{-2\delta}q^{2k})xy + (1-q)e^{-2\delta}q^{2k}(x^2+y^2).
\]
Here and below, we write
\[
(a;q)_\infty := \prodd k0\infty(1-aq^k) \qmand (q)_\infty:=(q;q)_\infty,\mfa  a\in\R, q\in(-1,1).
\]
For more background, see \citet{szablowski12qwiener} and references therein.
\section{Asymptotic probability of large jumps}\label{sec:bigjumps}
Consider the event of the process having at least one big jump in $(0,1]$:
\begin{multline*}
A\topp q(\epsilon):= \ccbb{N\topp q((0,1],\epsilon)\ge 1}\\ 
= \ccbb{\exists t\in(0,1]: X\topp q_{t-}< -\frac2{\sqrt{1-q}}+\epsilon, X_t\topp q>\frac2{\sqrt{1-q}}-\epsilon}.
\end{multline*}
\begin{Prop}\label{prop:bigjumps}
For all $q\in(-1,1)$, 
\[
\lim_{\epsilon\downarrow0}\frac{\proba(A\topp q(\epsilon))}{\epsilon^3} = \alpha_q. 
\]
\end{Prop}
\begin{proof}
We fix $q\in(-1,1)$ and drop the sup-scripts of $A\topp q$, $X\topp q$ $p\topp q$ and $p_{s,t}\topp q$ for the sake of simplicity. 
We first approximate the event of interest $A(\epsilon)$ by events of the
discretized processes. Namely, introduce
\[
A_n(\epsilon):= \ccbb{X_{\frac i{2^n}}<-\frac2{\sqrt{1-q}}+\epsilon, X_{\frac{i+1}{2^n}}>\frac2{\sqrt{1-q}}-\epsilon, \mbox{ for some } i=0,\dots,2^n-1}.
\]
Let $\wt X\topp n$ be the discretized process of $X$ defined by setting $\wt X\topp n_{i/2^n} := X_{i/2^n}$ for $i=0,\dots,2^n$, and $\wt X\topp n_t$
piecewise constant over the interval $[i/2^n,(i+1)/2^n)$ for each $i=0,\dots,2^n-1$. One can show that $\wt X\topp n\to X$ as $n\to\infty$ in $D([0,1])$
 almost surely \citep[page 151, Problem 12]{ethier86markov}. Therefore, by the dominated convergence theorem,
 \[
 \proba(A(\epsilon)) = \limn\proba(A_n(\epsilon)), \mfa \epsilon>0,
 \]
and it suffices to show
\equh\label{eq:limitAn}
\lim_{\epsilon\downarrow0}\frac1{\epsilon^3}\limn{
\proba(A_n(\epsilon))} = \alpha_q.
\eque
We prove this
 by the double-sum method (e.g.~\citep{aldous89probability,piterbarg96asymptotic}). Introduce
\[
B_{i,n}(\epsilon) := \ccbb{X_{\frac{i-1}{2^n}}<-\frac2{\sqrt{1-q}}+\epsilon, X_{\frac i{2^n}}>\frac2{\sqrt{1-q}}-\epsilon}, i=1,\dots,2^n.
\]
Then $A_n(\epsilon)=\bigcup_{i=1}^{2^n}B_{i,n}(\epsilon)$, and
\[
\summ i1{2^n}\proba(B_{i,n}(\epsilon)) - \sum_{\substack{1\leq i,j\leq{2^n}\\ i\neq j}}\proba(B_{i,n}(\epsilon)\cap B_{j,n}(\epsilon)) \leq P(A_n(\epsilon))\leq \summ i1{2^n}\proba(B_{i,n}(\epsilon)).
\]
Now to prove~\eqref{eq:limitAn}, it suffices to establish
\equh\label{eq:Bin}
\lim_{\epsilon\downarrow0} \frac1{\epsilon^3}\limn\summ i1{2^n}\proba(B_{i,n}(\epsilon)) = \alpha_q,
\eque
and
\equh
\lim_{\epsilon\downarrow0} \frac1{\epsilon^3}\limsupn\sum_{i\neq j}\proba(B_{i,n}(\epsilon)\cap B_{j,n}(\epsilon)) = 0.\label{eq:Bijn}
\eque
We first show~\eqref{eq:Bin}. By stationarity it is equivalent to compute $2^n\proba(B_{1,n}(\epsilon))$. Write $\delta:=1/2^n$. By definition,
\begin{multline*}
\proba(B_{1,n}(\epsilon)) = \int_{-\frac2{\sqrt{1-q}}}^{-\frac2{\sqrt{1-q}}+\epsilon}p(y_1)\int_{\frac2{\sqrt{1-q}}-\epsilon}^{\frac2{\sqrt{1-q}}}p_{0,\delta}(y_1,y_2)dy_2dy_1\\
= \int_{-\frac2{\sqrt{1-q}}}^{-\frac2{\sqrt{1-q}}+\epsilon}\int_{\frac2{\sqrt{1-q}}-\epsilon}^{\frac2{\sqrt{1-q}}}p(y_1)p(y_2) {(e^{-2\delta};q)_\infty \prodd k0\infty\frac1{\varphi_{q,k}(\delta,y_1,y_2)}}dy_2dy_1.
\end{multline*}
It is clear that $(e^{-2\delta};q)_\infty\sim 2\delta\cdot(q)_\infty$ as $\delta\downarrow0$, and 
\begin{multline*}
\lim_{\delta\downarrow0}\varphi_{q,k}(\delta,y_1,y_2) \\
=  (1-q^{2k})^2-(1-q)q^k(1+q^{2k})y_1y_2+(1-q)q^{2k}(y_1^2+y_2^2) =:\psi_{q,k}(y_1,y_2).
\end{multline*}
We have shown in \citep{bryc16local} that
\equh\label{eq:phi1}
\min_{|x|,|y|\leq \frac2{\sqrt{1-q}}}\varphi_{q,k}(\delta,x,y) = (1-e^{-\delta}q^k)^4\geq (1-|q|^k)^4, k\in\N_0, \delta>0,
\eque
and
\equh\label{eq:phi2}
\varphi_{q,0}(\delta,x,y) 
\ge  e^{-2\delta}\bb{16\sinh^4(\delta/2) + (1-q)(x-y)^2}.
\eque
It follows from the dominated convergence theorem, applied to the logarithm of the infinite product, that
\[
\lim_{\delta\downarrow0}\prodd k0\infty\frac1{\varphi_{q,k}(\delta,y_1,y_2)}
=\prodd k0\infty \frac1{\psi_{q,k}(y_1,y_2)}
=:\Psi_q(y_1,y_2).
\]
By inequalities~\eqref{eq:phi1},~\eqref{eq:phi2} and $(e^{-2\delta};q)_\infty\leq \prodd k0\infty(1+|q|^k) = (-1;|q|)_\infty$, we obtain 
\[
 \sup_{\substack{y_1\leq -2/\sqrt{1-q}+\epsilon\\ y_2\geq 2/\sqrt{1-q}-\epsilon}}{(e^{-2\delta};q)_\infty \prodd k0\infty\frac1{\varphi_{q,k}(\delta,y_1,y_2)}}
 \leq \frac{(-1;|q|)_\infty\cdot e^2}{4(1-q)(\frac2{\sqrt{1-q}}-\epsilon)^2(|q|)_\infty^4}<\infty
\]
for all $n\leq \N, \epsilon<2/\sqrt{1-q}$. Therefore, by the dominated convergence theorem again, for fixed $\epsilon\in(0,2/\sqrt{1-q})$,
\[
\limn 2^n\proba(B_{1,n}(\epsilon)) = 2\cdot(q)_\infty\int_{-\frac2{\sqrt{1-q}}}^{-\frac2{\sqrt{1-q}}+\epsilon}\int_{\frac2{\sqrt{1-q}}-\epsilon}^{\frac2{\sqrt{1-q}}}p(y_1)p(y_2) \Psi_q(y_1,y_2)dy_2dy_1.
\]
Observe that as $y_1\downarrow-2/\sqrt{1-q}$ and $y_2\uparrow 2/\sqrt{1-q}$,
\[
\Psi_q(y_1,y_2) \uparrow
 \prodd k0\infty\frac1{(1-q^{2k})^2+4q^k(1+q^{2k})+8q^{2k}}= \prodd k0\infty\frac{1}{(1+q^k)^4}.
\]
It then follows that as $\epsilon\downarrow0$,
\equh\label{eq:B1n:2}
\limn 2^n\proba(B_{1,n}(\epsilon))\sim \frac18\prodd k1\infty\frac{(1-q^k)}{(1+q^k)^4}\pp{\int_{-\frac2{\sqrt{1-q}}}^{-\frac2{\sqrt{1-q}}+\epsilon}p(y)dy}^2
\eque
with
\begin{multline}\label{eq:B1n:3}
\int_{-\frac2{\sqrt{1-q}}}^{-\frac2{\sqrt{1-q}}+\epsilon}p(y)dy\\
 = \frac{\sqrt{1-q}\cdot(q)_\infty}{2\pi}\int_{-\frac2{\sqrt{1-q}}}^{-\frac2{\sqrt{1-q}}+\epsilon}\sqrt{4-(1-q)y^2}\prodd k1\infty \bb{(1+q^k)^2-(1-q)y^2q^k}dy\\
\sim  \frac{(q)^3_\infty}{2\pi}\int_{-2}^{-2+\epsilon\sqrt{1-q}}\sqrt{4-y^2}dy\sim \frac{(q)_\infty^3}{2\pi}\frac43\pp{\epsilon\sqrt{1-q}}^{3/2}.
\end{multline}
Combining~\eqref{eq:B1n:2} and~\eqref{eq:B1n:3},~\eqref{eq:Bin} follows.

Finally, we show~\eqref{eq:Bijn}. Observe that for all $\ell\ge 2$,
\begin{align*}
\proba ( & B_{0,n}(\epsilon) \cap  B_{\ell,n}(\epsilon))
 =  \int_{-\frac2{\sqrt{1-q}}}^{-\frac2{\sqrt{1-q}}+\epsilon} p(y_1)\int_{\frac2{\sqrt{1-q}}-\epsilon}^{\frac2{\sqrt{1-q}}}p_{0,\delta}(y_1,y_2)\\
  & \quad \times \int_{-\frac2{\sqrt{1-q}}}^{-\frac2{\sqrt{1-q}}+\epsilon} p_{0,(\ell-1)\delta}(y_2,y_3)\int_{\frac2{\sqrt{1-q}}-\epsilon}^{\frac2{\sqrt{1-q}}}p_{0,\delta}(y_3,y_4)dy_4dy_3dy_2dy_1\\
  & = \int_{-\frac2{\sqrt{1-q}}}^{-\frac2{\sqrt{1-q}}+\epsilon}\int_{\frac2{\sqrt{1-q}}-\epsilon}^{\frac2{\sqrt{1-q}}} p(y_1)p_{0,\delta}(y_1,y_2) {(e^{-2(\ell-1)\delta};q)_\infty}\prodd k0\infty\frac1{\varphi_{q,k}((\ell-1)\delta,y_2,y_3)}\\
  & \quad \times \int_{-\frac2{\sqrt{1-q}}}^{-\frac2{\sqrt{1-q}}+\epsilon}\int_{\frac2{\sqrt{1-q}}-\epsilon}^{\frac2{\sqrt{1-q}}} p(y_3)p_{0,\delta}(y_3,y_4)dy_4dy_3dy_2dy_1.
\end{align*}
The infinite product along with $(e^{-2(\ell-1)\delta};q)_\infty$ above, in the domain of the integration for all $\epsilon<\epsilon_0$,  by~\eqref{eq:phi1} and~\eqref{eq:phi2} is uniformly bounded by
\[
(-1;|q|)_\infty\cdot \frac{e^{2(\ell-1)\delta}}{(1-q)(y_3-y_2)^2+16\sinh^4((\ell-1)\delta/2)}\frac1{(|q|)_\infty^4}\leq C_{q,\epsilon_0}e^{2(\ell-1)\delta}
\]
 for some constant $C_{q,\epsilon_0}$. Since $B_{i,n}(\epsilon)\cap B_{i+1,n}(\epsilon)=\emptyset$ for $\epsilon<2/\sqrt{1-q}$,  we have 
\begin{multline*}
\sum_{\substack{1\leq i,j\leq 2^n\\i\neq j}}\proba(B_{i,n}(\epsilon)\cap B_{j,n}(\epsilon)) \\
\leq C_{q,\epsilon_0} \proba(B_{1,n}(\epsilon))^2\sum_{1\leq i,j\leq 2^n} e^{(|i-j|-1)/2^n}\leq C'_{q,\epsilon_0}[2^n\proba(B_{1,n})]^2
\end{multline*}
for another constant $C'_{q,\epsilon_0}<\infty$ uniformly in $n\in\N$ and $\epsilon\leq \epsilon_0$.
Now,~\eqref{eq:Bijn} follows from~\eqref{eq:Bin}. The proof is thus completed.
\end{proof}
\section{A Poisson limit theorem}\label{sec:Poisson}
Consider Bernoulli random variables indicating whether the process has at least one big jump in each interval $(i-1,i]$:
\[
J_i\topp q(\epsilon) :=  \ind_{\ccbb{N\topp q((i-1,i],\epsilon)\ge 1}}, i\in\N.
\]
In this way, $\{J_i\topp q(\epsilon)\}_{i\in\N}$ is a stationary sequence of Bernoulli random variables with success rate $\proba(J_i\topp q(\epsilon) = 1) = \proba(A\topp q(\epsilon))$.  Set
\[
W\topp q_\epsilon := \summ i1{T_\epsilon}J_i\topp q(\epsilon) \qmwith T_\epsilon := \ceil{\frac1{\epsilon^3}}, \epsilon>0.
\]
We first prove the following.
\begin{Prop}\label{prop:poisson}
For all $q\in(-1,1)$, as $\epsilon\downarrow 0$, the distribution of $W\topp q_\epsilon$ converges to the Poisson distribution with parameter $\alpha_q$.
\end{Prop}
We will apply the Poisson limit theorem for dependent random variables established by \citet{chen75poisson}.
To do so, we start by investigating the mixing dependence of the $q$-Ornstein--Uhlenbeck process. For this purpose we write
\[
\calJ_1 :=\sigma\pp{J_1\topp q(\epsilon)} \qmand \calJ_m^\infty:=\sigma\pp{\ccbb{J_k\topp q(\epsilon):k\ge m}}, m\in\N.
\]
\begin{Lem}\label{lem:mixing}
For all $q\in(-1,1)$, there exists a constant $C_q<\infty$, such that 
\equh\label{eq:exponential}
\abs{\frac{p_{0,t}\topp q(x,y)}{p\topp q(y)}-1}\leq C_qe^{-t} \mfa t\ge 1, x,y\in(-2/\sqrt{1-q},2/\sqrt{1-q}).
\eque
As a consequence, the sequence $\{J_i\topp q(\epsilon)\}_{i\in\N}$ is $\psi$-mixing in the sense that 
\equh\label{eq:psimixing}
|\proba(B\mid\calJ_1)-\proba(B)|\leq C_qe^{-m}\proba(B)\quad\quad \mfa B\in\calJ_{m+1}^\infty, m\in\N.
\eque
\end{Lem}
\begin{proof}
 It is shown in \citep[Proposition 1, vii]{szablowski11structure} that
\equh\label{eq:szablowski}
C(x,e^{-t},q)\leq \frac{p_{0,t}\topp q(x,y)}{p\topp q(y)}\leq \frac{(e^{-2t};q)_\infty}{(e^{-t};q)_\infty^4} 
\eque
with
\[
C(x,\rho,q) = \frac{(\rho^2;q)_\infty}{\prodd k0\infty[(1+\rho^2q^{2k})^2+2(1-q)(1+\rho^2q^{k})|x\rho q^{2k}|+(1-q)\rho^2x^2q^{2k}]}.
\]
Observe that as $\rho\downarrow 0$, 
\[
(\rho;q)_\infty  = \exp\ccbb{\summ k0\infty\log(1-\rho q^k)} = \exp \ccbb{-\frac\rho{1-q}+O(\rho^2)} 
= 1-\frac\rho{1-q}+O(\rho^2),
\]
so 
\equh\label{eq:rhs}
\frac{(\rho^2;q)_\infty}{(\rho;q)_\infty^4}-1 = \frac{4\rho}{1-q}+O(\rho^2).
\eque
At the same time, bounding $\rho^mq^k$ by $\rho|q|^k$ from above for all $m,k\in\N$ in  the denominator of $C(x,\rho,q)$, we have
\equh\label{eq:lhs}
C(x,\rho,q)\geq \frac{(\rho^2;q)_\infty}{\prodd k0\infty[1+(7+8\sqrt 2)\rho q^k]} = 1-\frac{(7+8\sqrt 2)\rho}{1-|q|}+O(\rho^2).
\eque
Combining~\eqref{eq:rhs} and~\eqref{eq:lhs}, it follows that there exists $t_0$ large, such that for all $t>t_0$, 
\[
\abs{\frac{p_{0,t}\topp q(x,y)}{p\topp q(y)}-1}\leq Ce^{-t}
\]
for some large constant $C$. This and~\eqref{eq:szablowski} yield that the inequality remain true for all $t\ge 1$, by possibly increasing the value of the constant.

Now we prove the second part of the result. We drop the sup-scripts in $X\topp q$, $p\topp q$ and $p_{s,t}\topp q$ for the sake of simplicity. Recall that $X$ is a Markov process, and note that $\calX_1 := \sigma(\{X_t,0\leq t\leq 1\})\supset\calJ_1$. Then, we  observe that for all $B\in\calJ_{m+1}^\infty, m\in\N$,
\begin{multline*}
|\proba(B\mid\calJ_1)-\proba(B)| = |\esp[\proba(B\mid\calX_1)-\proba(B)\mid\calJ_1]|\\
=\abs{\esp\bb{\proba\pp{B\mid X_1}-\proba(B)\mmid\calJ_1}} \leq \esp\bb{\abs{\proba\pp{B\mid X_1}-\proba(B)}\mmid\calJ_1},
\end{multline*}
and that~\eqref{eq:exponential} yields a uniform upper bound 
\begin{multline*}
\abs{\proba(B\mid X_1= x) - \proba(B)}\\
 \leq \int \proba(B\mid X_m = y)p(y)\abs{\frac{p_{0,m-1}(x,y)}{p(y)}-1}dy 
\leq \proba(B)C_qe^{-m}.
\end{multline*}
The desired result follows.
\end{proof}
\begin{proof}[Proof of Proposition~\ref{prop:poisson}]
Fix $q\in(-1,1)$. For the sake of simplicity, we omit the sup-scripts in $W\topp q$ and $J\topp q_i$. 
Let $Z$ be a Poisson random variable with parameter $\alpha_q$. We  apply \citet[Theorem 4.1]{chen75poisson}, which says the following: if $\{J_i(\epsilon)\}_{i\in\N}$ satisfies
\equh\label{eq:phimixing}
|\proba(B\mid\calJ_1)-\proba(B)|\leq \phi(m) \quad\quad \mfa B\in\calJ_{m+1}^\infty, m\in\N,
\eque
then for all continuous function $h$ with $|h|\leq 1$ and $m\in\N$, 
\begin{multline}\label{eq:chen}
|\esp h(W_\epsilon) - \esp h(Z)| 
\leq 6\pp{\frac1{\sqrt {\lambda_\epsilon}}\wedge 1}\\
\times \bb{\var(W_\epsilon)-{\lambda_\epsilon}+ 2(2m+1)\summ i1{T_\epsilon}(\esp J_i(\epsilon))^2 + 4({\lambda_\epsilon}+1)T_\epsilon\phi(m+1)},
\end{multline}
with $\lambda_\epsilon := \summ i1{T_\epsilon}\esp J_i(\epsilon) = T_\epsilon\esp J_1(\epsilon)$.

The property~\eqref{eq:phimixing} is known as $\phi$-mixing if $\limm\phi(m) = 0$. Clearly, it is  weaker than the $\psi$-mixing property~\eqref{eq:psimixing} established in Lemma~\ref{lem:mixing} (see~\citep{bradley05basic} for more background on mixing conditions). As a consequence,~\eqref{eq:phimixing} now holds with 
\[
\phi(m) = C_qe^{-m}, m\in\N. 
\]
With this choice of $\phi$, we can choose $m = m_\epsilon$ such that the right-hand side of~\eqref{eq:chen} vanishes as $\epsilon\downarrow 0$.
Indeed, since we have shown that 
\equh\label{eq:1}
\lim_{\epsilon\downarrow0} \lambda_\epsilon =  \lim_{\epsilon\downarrow0}T_\epsilon\proba(A(\epsilon)) = \alpha_q,
\eque 
it suffices to take $m_\epsilon = \ceil{\epsilon^{-3+\delta}}$ for any $\delta>0$ to make the second and last terms in the right-hand side of~\eqref{eq:chen} vanish. Therefore, the desired convergence Poisson limit theorem will follow from
\equh\label{eq:var}
\lim_{\epsilon\downarrow0}\var(W_\epsilon) = \alpha_q.
\eque
To show this, write
\equh\label{eq:var1}
\var(W_\epsilon) = \summ i1{T_\epsilon}\var\pp{J_i(\epsilon)} + 2\frac1{T_\epsilon}\summ \ell2{T_\epsilon}\pp{1-\frac{\ell-1}{T_\epsilon}}T_\epsilon^2\cov\pp{J_1(\epsilon),J_\ell(\epsilon)}.
\eque
It follows from~\eqref{eq:1} that  the first summand converges to $\alpha_q$.
By~\eqref{eq:psimixing},
\begin{multline*}
\abs{\cov\pp{J_1(\epsilon),J_\ell(\epsilon)}}  = \abs{\proba(J_1(\epsilon) = 1, J_\ell(\epsilon) = 1) - \proba(J_1(\epsilon) = 1)\proba(J_\ell(\epsilon) = 1)} \\
\leq\psi(\ell-1)\proba(J_1(\epsilon) = 1)\proba(J_\ell(\epsilon) = 1) = C_qe^{-(\ell-1)}(\esp J_1(\epsilon))^2 \mfa\ell\ge 2.
\end{multline*}
Combining this and~\eqref{eq:1}, the second term in~\eqref{eq:var1} converges to zero as $\epsilon\downarrow0$. We have established~\eqref{eq:var} and the desired result.
\end{proof}

\begin{proof}[Proof of Theorem~\ref{thm:1}]
We omit the sup-script in $N\topp q$.
First, writing
$N((0,T_\epsilon],\epsilon) = \summ i1{T_\epsilon}N((i-1,i],\epsilon)$, 
we have
\[
\proba(N((0,T_\epsilon]),\epsilon)\neq W_\epsilon) \leq T_\epsilon \proba(N((0,1],\epsilon)\ge 2).
\]
As in the proof of Proposition~\ref{prop:bigjumps}, using the same notation of $B_{i,n}$,
\[
\lim_{\epsilon\downarrow0}T_\epsilon\proba(N((0,1],\epsilon)\ge 2) \leq \lim_{\epsilon\downarrow0} T_\epsilon\limsupn \sum_{i\neq j}\proba(B_{i,n}(\epsilon)\cap B_{j,n}(\epsilon)) = 0
\]
by~\eqref{eq:Bijn}. Finally, $\proba(N(0,T_\epsilon],\epsilon)\neq N((0,\epsilon^{-3}],\epsilon)) \leq \proba(N((0,1],\epsilon)>0)\to 0$ as $\epsilon\downarrow 0$. The desired result thus follows.
\end{proof}
\subsection*{Acknowledgements} The author would like to thank Wlodek Bryc for many inspiring and helpful discussions, and for his careful reading of the paper. The author's research was partially supported by NSA grant  H98230-14-1-0318.

\bibliographystyle{apalike}
\bibliography{references}

\end{document}